\documentclass[11pt,a4paper]{article}

\pdfoutput=1 
\usepackage{epsf,epsfig,amsfonts,amsgen,amsmath,amstext,amsbsy,amsopn,amsthm,cases,listings,color
}
\usepackage{ebezier,eepic}
\usepackage{color}
\usepackage{multirow}
\usepackage{epstopdf}
\usepackage{graphicx}
\usepackage{pgf,tikz}
\usepackage{mathrsfs}
\usepackage[marginal]{footmisc}
\usepackage{enumitem}
\usepackage[titletoc]{appendix}
\usepackage{booktabs}
\usepackage{url}
\usepackage{mathtools}

\usepackage{pgfplots}
\usepackage{authblk}
\usepackage{amssymb}

\usepackage{wasysym}

\usepackage{empheq}

\usepackage{dsfont}
\usepackage{verbatim}

\usepackage{caption}
\usepackage{subcaption}
\pgfplotsset{compat=1.18}
\usepackage{mathrsfs}
\usepackage{wasysym} 
\usetikzlibrary{arrows}
\usepackage{aligned-overset}
\usepackage{bm}
\usepackage[backref=page]{hyperref}
\usepackage{makecell}

\newcommand{\Mod}[1]{\ (\mathrm{mod}\ #1)}

\allowdisplaybreaks[1]

\definecolor{uuuuuu}{rgb}{0.27,0.27,0.27}
\definecolor{sqsqsq}{rgb}{0.1255,0.1255,0.1255}

\newtheorem{definition}{Definition} [section]
\newtheorem{theorem}[definition]{Theorem}
\newtheorem{lemma}[definition]{Lemma}

\newtheorem{claim}[definition]{Claim}

\newtheorem{fact}[definition]{Fact}
\newenvironment{pf}{\noindent{\bf Proof}}

\begin{document}
\title{\bf\Large Andr{\'a}sfai--Erd\H{o}s--S\'{o}s theorem under max-degree constraints}
\date{\today}
\author[1]{Xizhi Liu\thanks{Email: \texttt{liuxizhi@ustc.edu.cn}}}
\author[2]{Sijie Ren\thanks{Email: \texttt{rensijie1@126.com}}}
\author[2]{Jian Wang\thanks{Email: \texttt{wangjian01@tyut.edu.cn}}}
\affil[1]{{\small School of Mathematical Sciences, University of Science and Technology of China, Hefei, China}}
\affil[2]{{\small Department of Mathematics, 
            Taiyuan University of Technology, Taiyuan, China}}
%
\maketitle
\begin{abstract}
    We establish the following strengthening of the celebrated Andr{\'a}sfai--Erd\H{o}s--S\'{o}s theorem: If $G$ is an $n$-vertex $K_{r+1}$-free graph whose minimum degree $\delta(G)$ and maximum degree $\Delta(G)$ satisfy 
    \begin{align*}
        \delta(G) > \min \left\{ \frac{3r-4}{3r-2}n-\frac{\Delta(G)}{3r-2},~n-\frac{\Delta(G)+1}{r-1} \right\}, 
    \end{align*}
    then $G$ is $r$-partite. 
    This bound is tight for all feasible values of $\Delta(G)$. 
    We also obtain an analogous tight result for graphs with large odd girth. 

    Our proof does not rely on the Andr{\'a}sfai--Erd\H{o}s--S\'{o}s theorem itself, and therefore yields an alternative proof of this classical result.
\end{abstract}

\section{Introduction}\label{SEC:Introduction}
Given a family $\mathcal{F}$ of graphs, we say a graph $G$ is $\mathcal{F}$-free if it does not contain any member of $\mathcal{F}$ as a subgraph. 
Let $K_{r}$ and $C_{r}$ denote the complete graph and the cycle on $r$ vertices, respectively. 
Extending the seminal theorem of Tur\'{a}n~\cite{TU41}, Andr{\'a}sfai--Erd\H{o}s--S\'{o}s~\cite{AES74} established the following classical result. 

\begin{theorem}[\cite{AES74}]\label{THM:AES-clique}
    Let $r \ge 2$ be an integer. 
    Every $K_{r+1}$-free graph on $n$ vertices with minimum degree greater than $\frac{3r-4}{3r-1}n$ is $r$-partite.
\end{theorem}

Let $C_{\le 2k+1} \coloneqq \{C_3, C_5, \cdots, C_{2k+1}\}$ denote the collection of all odd cycles of length at most $2k+1$.  
The following theorem was also established in~\cite{AES74}.  
 
\begin{theorem}[\cite{AES74}]\label{THM:AES-odd-cycle}
    Every $C_{\le 2k+1}$-free graph on $n$ vertices with minimum degree greater than $\frac{2}{2k+3}n$ is bipartite.
\end{theorem}

In this note, we study the affect of a large degree vertex on the two theorems of Andr{\'a}sfai--Erd\H{o}s--S\'{o}s. 
Our motivation partly comes from a classical result of Balister--Bollob{\'a}s--Riordan--Schelp~\cite{BBRS03}, which, extending the seminal theorems of Mantel~\cite{Man07} and Simonovits~\cite{Sim69} (see also~{\cite[p.~150]{Bol78}}), shows that for large $n$, the maximum number of edges in an $n$-vertex $C_{2k+1}$-free graph with maximum degree $\Delta \in [n/2,~n-k-1]$ is achieved by the complete bipartite graph with one part of size $\Delta$. 
A similar statement for $K_{r+1}$-free graph follows easily from Erd\H{o}s' degree-majorization algorithm~\cite{Erd70} (see also~\cite{Fur15}). 
More recently, related questions have been investigated in~\cite{HY22} and~\cite{HHLLYZ23c} for edge-critical graphs and hypergraphs, with further applications to matching-type problems~\cite{HLLYZ23,HHLLYZ25deg,DHLY25}. 

For a graph $G$, let $\delta(G)$ and $\Delta(G)$ denote the minimum and maximum degree of $G$, respectively.
The main results of this note are as follows. 

\begin{theorem}\label{THM:max-deg-AES-clique}
    Let $r \ge 2$ be an integer. 
    Suppose that $G$ is an $n$-vertex $K_{r+1}$-free graph satisfying 
    \begin{align*}
        \delta(G) 
        > \min \left\{ \frac{3r-4}{3r-2}n-\frac{\Delta(G)}{3r-2},~n-\frac{\Delta(G)+1}{r-1}\right\}. 
    \end{align*}
    Then $G$ is $r$-partite.
\end{theorem}

\begin{theorem}\label{THM:max-deg-AES-odd-cycle}
    Let $k \ge 2$ be an integer. 
    Suppose that $G$ is an $n$-vertex $C_{\le 2k+1}$-free graph satisfying 
    \begin{align*}
        \delta(G) > \min \left\{ \frac{n}{k+1}-\frac{\Delta(G)}{2k+2},~\frac{n-1-\Delta(G)}{k} \right\}. 
    \end{align*}
    Then $G$ is bipartite. 
\end{theorem}

It is straightforward to verify that Theorems~\ref{THM:max-deg-AES-clique} and~\ref{THM:max-deg-AES-odd-cycle} imply the corresponding theorems of Andr{\'a}sfai--Erd\H{o}s--S\'{o}s. 
We note that our proofs differ from those of Andr{\'a}sfai--Erd\H{o}s--S\'{o}s and the alternative elegant short proof of Brandt~\cite{Bra03}, and they do not rely on Theorems~\ref{THM:AES-clique} and~\ref{THM:AES-odd-cycle}. Thus, they provide alternative proofs for both theorems of Andr{\'a}sfai--Erd\H{o}s--S\'{o}s.  
Moreover, the bounds in the above theorems are tight, as demonstrated by the following constructions.
For simplicity, we omit floor and ceiling signs below.

Let $x_1, \ldots, x_k$ be positive integers. 
For a graph $H$ on vertex set $[k]$, the \emph{blowup} $H[x_1, \ldots, x_k]$ is the graph obtained from $H$ by replacing each vertex $i$ with a set of size $x_i$, and replacing each edge $ij$ with the corresponding complete bipartite graph.

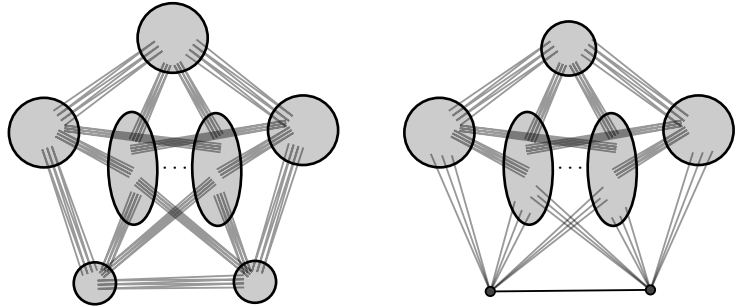
\begin{figure}[htbp]
\centering
\tikzset{every picture/.style={line width=.7pt}} 
\begin{tikzpicture}[x=0.75pt,y=0.75pt,yscale=-0.7,xscale=0.7]
\draw [line width=1pt, fill=black, fill opacity=0.2] (158.92,46.89) .. controls (158.92,33.08) and (170.11,21.89) .. (183.92,21.89) .. controls (197.72,21.89) and (208.92,33.08) .. (208.92,46.89) .. controls (208.92,60.69) and (197.72,71.89) .. (183.92,71.89) .. controls (170.11,71.89) and (158.92,60.69) .. (158.92,46.89) -- cycle ;
\draw  [line width=1pt, fill=black, fill opacity=0.2] (251.92,113.15) .. controls (251.92,99.35) and (263.11,88.15) .. (276.92,88.15) .. controls (290.73,88.15) and (301.92,99.35) .. (301.92,113.15) .. controls (301.92,126.96) and (290.73,138.15) .. (276.92,138.15) .. controls (263.11,138.15) and (251.92,126.96) .. (251.92,113.15) -- cycle ;
\draw  [line width=1pt, fill=black, fill opacity=0.2] (67.15,114.86) .. controls (67.15,101.06) and (78.35,89.86) .. (92.15,89.86) .. controls (105.96,89.86) and (117.15,101.06) .. (117.15,114.86) .. controls (117.15,128.67) and (105.96,139.86) .. (92.15,139.86) .. controls (78.35,139.86) and (67.15,128.67) .. (67.15,114.86) -- cycle ;
\draw  [line width=1pt, fill=black, fill opacity=0.2] (113.39,223.14) .. controls (113.39,214.83) and (120.13,208.08) .. (128.45,208.08) .. controls (136.76,208.08) and (143.5,214.83) .. (143.5,223.14) .. controls (143.5,231.46) and (136.76,238.2) .. (128.45,238.2) .. controls (120.13,238.2) and (113.39,231.46) .. (113.39,223.14) -- cycle ;
\draw  [line width=1pt, fill=black, fill opacity=0.2] (227.58,222.08) .. controls (227.58,213.77) and (234.32,207.03) .. (242.64,207.03) .. controls (250.95,207.03) and (257.69,213.77) .. (257.69,222.08) .. controls (257.69,230.4) and (250.95,237.14) .. (242.64,237.14) .. controls (234.32,237.14) and (227.58,230.4) .. (227.58,222.08) -- cycle ;
\draw [line width=1pt, fill=black, fill opacity=0.2]  (156.37,181.04) .. controls (146.73,181.25) and (138.51,163.28) .. (138.02,140.91) .. controls (137.53,118.55) and (144.96,100.24) .. (154.61,100.03) .. controls (164.25,99.82) and (172.47,117.78) .. (172.96,140.15) .. controls (173.45,162.52) and (166.02,180.82) .. (156.37,181.04) -- cycle ;
\draw [color=black, opacity=0.4]  (98.85,99.23) -- (171.23,49.25) -- (101.53,102.73) -- (173.91,52.75) -- (104.22,106.23) -- (176.59,56.24) -- (106.9,109.72) ;
\draw [color=black, opacity=0.4]  (200.91,47.95) -- (269.51,103) -- (198.33,51.52) -- (266.93,106.57) -- (195.75,55.09) -- (264.35,110.14) -- (193.17,58.66) ;
\draw  [color=black, opacity=0.4] (277.11,125.51) -- (247.55,216.01) -- (273.66,124.53) -- (244.09,215.03) -- (270.2,123.54) -- (240.64,214.04) -- (266.75,122.56) ;
\draw [color=black, opacity=0.4]  (234.96,226.16) -- (130.55,227.09) -- (234.85,223.21) -- (130.45,224.14) -- (234.74,220.26) -- (130.34,221.19) -- (234.63,217.31) ;
\draw [color=black, opacity=0.4]  (119.93,217.67) -- (90.96,126.98) -- (123.31,216.44) -- (94.34,125.75) -- (126.69,215.22) -- (97.72,124.53) -- (130.07,214) ;
\draw  [color=black, opacity=0.4] (153.1,119.59) -- (177.91,64.63) -- (155.39,120.51) -- (180.2,65.56) -- (157.68,121.43) -- (182.49,66.48) -- (159.98,122.36) ;
\draw  [color=black, opacity=0.4] (213.41,119.62) -- (185.04,66.42) -- (215.54,118.37) -- (187.17,65.17) -- (217.67,117.12) -- (189.3,63.92) -- (219.8,115.87) ;
\draw [color=black, opacity=0.4]  (153.03,147) -- (100.37,117.63) -- (154.14,144.79) -- (101.48,115.42) -- (155.26,142.58) -- (102.6,113.22) -- (156.37,140.38) ;
\draw [color=black, opacity=0.4]  (215.37,141.38) -- (267.4,110.91) -- (216.7,143.46) -- (268.73,112.99) -- (218.04,145.54) -- (270.07,115.07) -- (219.37,147.62) ;
\draw  [color=black, opacity=0.4] (161.42,159.83) -- (137.56,215.2) -- (159.11,158.95) -- (135.25,214.32) -- (156.8,158.06) -- (132.95,213.44) -- (154.5,157.18) ;
\draw [color=black, opacity=0.4]  (220.33,157.38) -- (239.61,214.51) -- (218.03,158.26) -- (237.3,215.39) -- (215.72,159.15) -- (234.99,216.28) -- (213.41,160.03) ;
\draw  [color=black, opacity=0.4] (214.89,151.85) -- (134.53,219) -- (213.44,150.19) -- (133.09,217.34) -- (211.99,148.53) -- (131.64,215.68) -- (210.55,146.87) ;
\draw [color=black, opacity=0.4]  (160.87,147.19) -- (240.64,215.04) -- (159.48,148.9) -- (239.25,216.75) -- (158.09,150.61) -- (237.86,218.45) -- (156.7,152.31) ;
\draw  [color=black, opacity=0.4] (218,129) -- (106.43,113.67) -- (218.24,127.02) -- (106.66,111.7) -- (218.47,125.05) -- (106.9,109.72) -- (218.71,123.07) ;
\draw  [color=black, opacity=0.4] (153.49,124.34) -- (264.64,106.22) -- (153.84,126.3) -- (265,108.18) -- (154.2,128.25) -- (265.35,110.14) -- (154.55,130.21) ;
\draw  [line width=1pt, fill=black, fill opacity=0.2] (216.25,181.88) .. controls (206.61,182.09) and (198.39,164.13) .. (197.9,141.76) .. controls (197.41,119.39) and (204.84,101.09) .. (214.49,100.88) .. controls (224.13,100.67) and (232.35,118.63) .. (232.84,141) .. controls (233.33,163.37) and (225.9,181.67) .. (216.25,181.88) -- cycle ;
\draw  [line width=1pt, fill=black, fill opacity=0.2] (446.92,54.44) .. controls (446.92,43.71) and (455.62,35) .. (466.36,35) .. controls (477.1,35) and (485.8,43.71) .. (485.8,54.44) .. controls (485.8,65.18) and (477.1,73.89) .. (466.36,73.89) .. controls (455.62,73.89) and (446.92,65.18) .. (446.92,54.44) -- cycle ;
\draw [line width=1pt, fill=black, fill opacity=0.2]  (533.92,113.15) .. controls (533.92,99.35) and (545.11,88.15) .. (558.92,88.15) .. controls (572.73,88.15) and (583.92,99.35) .. (583.92,113.15) .. controls (583.92,126.96) and (572.73,138.15) .. (558.92,138.15) .. controls (545.11,138.15) and (533.92,126.96) .. (533.92,113.15) -- cycle ;
\draw  [line width=1pt, fill=black, fill opacity=0.2] (349.15,114.86) .. controls (349.15,101.06) and (360.35,89.86) .. (374.15,89.86) .. controls (387.96,89.86) and (399.15,101.06) .. (399.15,114.86) .. controls (399.15,128.67) and (387.96,139.86) .. (374.15,139.86) .. controls (360.35,139.86) and (349.15,128.67) .. (349.15,114.86) -- cycle ;
\draw  [line width=1pt, fill=black, fill opacity=0.2] (438.37,181.04) .. controls (428.73,181.25) and (420.51,163.28) .. (420.02,140.91) .. controls (419.53,118.55) and (426.96,100.24) .. (436.61,100.03) .. controls (446.25,99.82) and (454.47,117.78) .. (454.96,140.15) .. controls (455.45,162.52) and (448.02,180.82) .. (438.37,181.04) -- cycle ;
\draw  [color=black, opacity=0.4] (380.85,99.23) -- (453.23,49.25) -- (383.53,102.73) -- (455.91,52.75) -- (386.22,106.23) -- (458.59,56.24) -- (388.9,109.72) ;
\draw  [color=black, opacity=0.4] (482.91,47.95) -- (551.51,103) -- (480.33,51.52) -- (548.93,106.57) -- (477.75,55.09) -- (546.35,110.14) -- (475.17,58.66) ;
\draw  [color=black, opacity=0.4] (435.1,119.59) -- (459.91,64.63) -- (437.39,120.51) -- (462.2,65.56) -- (439.68,121.43) -- (464.49,66.48) -- (441.98,122.36) ;
\draw  [color=black, opacity=0.4] (495.41,119.62) -- (467.04,66.42) -- (497.54,118.37) -- (469.17,65.17) -- (499.67,117.12) -- (471.3,63.92) -- (501.8,115.87) ;
\draw  [color=black, opacity=0.4] (435.03,147) -- (382.37,117.63) -- (436.14,144.79) -- (383.48,115.42) -- (437.26,142.58) -- (384.6,113.22) -- (438.37,140.38) ;
\draw  [color=black, opacity=0.4] (497.37,141.38) -- (549.4,110.91) -- (498.7,143.46) -- (550.73,112.99) -- (500.04,145.54) -- (552.07,115.07) -- (501.37,147.62) ;
\draw  [color=black, opacity=0.4] (500,129) -- (388.43,113.67) -- (500.24,127.02) -- (388.66,111.7) -- (500.47,125.05) -- (388.9,109.72) -- (500.71,123.07) ;
\draw  [color=black, opacity=0.4] (435.49,124.34) -- (546.64,106.22) -- (435.84,126.3) -- (547,108.18) -- (436.2,128.25) -- (547.35,110.14) -- (436.55,130.21) ;
\draw [line width=1pt, fill=black, fill opacity=0.2]  (498.25,181.88) .. controls (488.61,182.09) and (480.39,164.13) .. (479.9,141.76) .. controls (479.41,119.39) and (486.84,101.09) .. (496.49,100.88) .. controls (506.13,100.67) and (514.35,118.63) .. (514.84,141) .. controls (515.33,163.37) and (507.9,181.67) .. (498.25,181.88) -- cycle ;
\draw   (410.45,229) -- (524.64,227.88) ;
\draw  [color=black, opacity=0.4]  (368,130.12) -- (410.45,229) ;
\draw  [color=black, opacity=0.4]  (383,131.18) -- (410.45,229) ;
\draw  [color=black, opacity=0.4]  (428,164.09) -- (410.45,229) ;
\draw  [color=black, opacity=0.4]  (439,172.59) -- (410.45,229) ;
\draw  [color=black, opacity=0.4]  (497,175.77) -- (524.64,227.88) ;
\draw  [color=black, opacity=0.4]  (507,168.34) -- (524.64,227.88) ;
\draw  [color=black, opacity=0.4]  (555,129.06) -- (524.64,227.88) ;
\draw  [color=black, opacity=0.4]  (569,128) -- (524.64,227.88) ;
\draw  [color=black, opacity=0.4]  (376,131.18) -- (410.45,229) ;
\draw  [color=black, opacity=0.4]  (433,170.46) -- (410.45,229) ;
\draw [color=black, opacity=0.4]   (503,174.71) -- (524.64,227.88) ;
\draw  [color=black, opacity=0.4]  (562,129.06) -- (524.64,227.88) ;
\draw [color=black, opacity=0.4]   (491,154) -- (410.45,229) ;
\draw  [color=black, opacity=0.4]  (489,162) -- (410.45,229) ;
\draw  [color=black, opacity=0.4]  (494,163) -- (410.45,229) ;
\draw  [color=black, opacity=0.4]  (443,153) -- (524.64,227.88) ;
\draw [color=black, opacity=0.4]   (444,160) -- (524.64,227.88) ;
\draw  [color=black, opacity=0.4]  (439,161) -- (524.64,227.88) ;

\draw (174,135) node [anchor=north west][inner sep=0.75pt]   [align=left] {\scriptsize{$\cdots$}};
\draw (456,135) node [anchor=north west][inner sep=0.75pt]   [align=left] {\scriptsize{$\cdots$}};
\draw [fill=uuuuuu] (410.45,229) circle (2.5pt);
\draw [fill=uuuuuu] (524.64,227.88) circle (2.5pt);
\end{tikzpicture}
\caption{Extremal constructions for Theorem~\ref{THM:max-deg-AES-clique}.} 
\label{Fig:W53}
\end{figure}

Let $W_{r}$ denote the graph with vertex set $[r+3]$ and edge set 
\begin{align*}
    \{12, 23, 34, 45, 51\}
    \cup \{ij \colon i \in [5],~j \in [6, r+3]\}
    \cup \{ij \colon \{i,j\} \subseteq [6, r+3]\}, 
\end{align*}
that is, $W_{r}$ is join of a $5$-cycle $C_{5}$ and a complete graph $K_{r-2}$. 

The extremal construction for Theorem~\ref{THM:max-deg-AES-clique} in the case $\Delta(G) < \frac{2r-2}{2r-1}n$ is given by the blowup $W_r[x_1, \ldots, x_{r+3}]$ (see the left illustration in Figure~\ref{Fig:W53}), where $\alpha \coloneqq \frac{\Delta(G)}{n} - \frac{3r-4}{3r-1}$ (noting that Theorem~\ref{THM:max-deg-AES-clique} implicitly requires $\Delta(G) > \frac{3r-4}{3r-1} n$), and 
\begin{align*}
    x_1 = x_2 = x_5 & \coloneqq \left(\tfrac{1}{3r-1}+\tfrac{r \alpha}{3r-2}\right)n, \\
    x_3 = x_4 & \coloneqq \left(\tfrac{1}{3r-1}-\tfrac{(2r-1)\alpha}{3r-2}\right)n, \\ 
    x_6 = \cdots = x_{r+3} & \coloneqq \left(\tfrac{3}{3r-1} + \tfrac{\alpha}{3r-2}\right)n. 
\end{align*}

The extremal construction for Theorem~\ref{THM:max-deg-AES-clique} in the case $\Delta \ge \frac{2r-2}{2r-1}n$ is given by the blowup  $W_r[y_1, \ldots, y_{r+3}]$ (see the right illustration in Figure~\ref{Fig:W53}), where $\beta \coloneqq \frac{\Delta(G)}{n} - \frac{2r-2}{2r-1}$, and 
\begin{align*}
    y_1 & \coloneqq \left(\tfrac{1}{2r-1} - \beta\right)n-2, \\
    y_2 = y_5 & \coloneqq \left(\tfrac{1}{2r-1} + \tfrac{\beta}{2r-2}\right)n, \\
    y_3 = y_4 & \coloneqq 1, \\
    y_6 = \cdots = y_{r+3} & \coloneqq \left(\tfrac{2}{2r-1} + \tfrac{\beta}{r-1}\right)n. 
\end{align*}

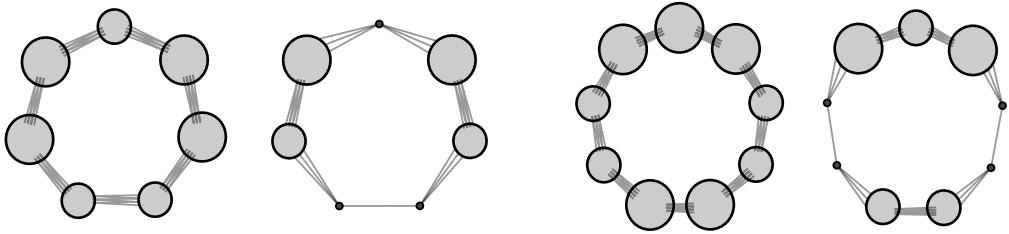
\begin{figure}[htbp]
\centering
\tikzset{every picture/.style={line width=.7pt}} 
\begin{tikzpicture}[x=0.75pt,y=0.75pt,yscale=-0.85,xscale=0.85]
\draw [line width=1pt, fill=black, fill opacity=0.2]  (129.57,44.42) .. controls (135.51,39.37) and (144.27,40.28) .. (149.14,46.45) .. controls (154,52.63) and (153.13,61.73) .. (147.19,66.79) .. controls (141.24,71.84) and (132.48,70.94) .. (127.62,64.76) .. controls (122.76,58.58) and (123.63,49.48) .. (129.57,44.42) -- cycle ;
\draw [line width=1pt, fill=black, fill opacity=0.2]  (140.21,90) .. controls (146.15,84.95) and (154.91,85.86) .. (159.77,92.03) .. controls (164.64,98.21) and (163.76,107.31) .. (157.82,112.37) .. controls (151.88,117.42) and (143.12,116.52) .. (138.26,110.34) .. controls (133.39,104.16) and (134.27,95.06) .. (140.21,90) -- cycle ;
\draw [line width=1pt, fill=black, fill opacity=0.2]  (38.82,91.43) .. controls (44.76,86.38) and (53.52,87.28) .. (58.39,93.46) .. controls (63.25,99.64) and (62.38,108.74) .. (56.43,113.8) .. controls (50.49,118.85) and (41.73,117.94) .. (36.87,111.77) .. controls (32.01,105.59) and (32.88,96.49) .. (38.82,91.43) -- cycle ;
\draw [line width=1pt, fill=black, fill opacity=0.2]  (48.27,45.57) .. controls (54.21,40.51) and (62.97,41.42) .. (67.83,47.6) .. controls (72.7,53.77) and (71.82,62.88) .. (65.88,67.93) .. controls (59.94,72.99) and (51.18,72.08) .. (46.31,65.9) .. controls (41.45,59.73) and (42.32,50.62) .. (48.27,45.57) -- cycle ;
\draw [line width=1pt, fill=black, fill opacity=0.2]  (115.2,130.42) .. controls (119.36,126.88) and (125.49,127.52) .. (128.89,131.84) .. controls (132.3,136.16) and (131.68,142.54) .. (127.52,146.08) .. controls (123.36,149.62) and (117.23,148.98) .. (113.83,144.66) .. controls (110.42,140.33) and (111.04,133.96) .. (115.2,130.42) -- cycle ;
\draw [line width=1pt, fill=black, fill opacity=0.2]  (70.07,131.06) .. controls (74.23,127.52) and (80.37,128.15) .. (83.77,132.48) .. controls (87.17,136.8) and (86.56,143.17) .. (82.4,146.71) .. controls (78.24,150.25) and (72.11,149.61) .. (68.71,145.29) .. controls (65.3,140.97) and (65.91,134.59) .. (70.07,131.06) -- cycle ;
\draw [line width=1pt, fill=black, fill opacity=0.2]  (91.3,28) .. controls (95.46,24.46) and (101.59,25.1) .. (104.99,29.42) .. controls (108.4,33.75) and (107.78,40.12) .. (103.63,43.66) .. controls (99.47,47.2) and (93.33,46.56) .. (89.93,42.24) .. controls (86.52,37.91) and (87.14,31.54) .. (91.3,28) -- cycle ;
\draw [color=black, opacity=0.4]  (105.38,35.05) -- (130.25,47.31) -- (104.68,36.9) -- (129.55,49.16) -- (103.98,38.76) -- (128.85,51.02) -- (103.28,40.62) ;
\draw [color=black, opacity=0.4]  (143.52,64.35) -- (147.92,92.6) -- (141.65,64.8) -- (146.05,93.05) -- (139.78,65.26) -- (144.18,93.51) -- (137.91,65.72) ;
\draw [color=black, opacity=0.4]  (144.96,111.61) -- (127.52,133.75) -- (143.39,110.45) -- (125.96,132.6) -- (141.82,109.3) -- (124.39,131.45) -- (140.25,108.15) ;
\draw [color=black, opacity=0.4]  (112.21,141.6) -- (84.72,140.11) -- (112.18,139.61) -- (84.68,138.12) -- (112.14,137.61) -- (84.65,136.12) -- (112.11,135.62) ;
\draw [color=black, opacity=0.4]  (68.08,136.13) -- (50.72,113.92) -- (69.48,134.76) -- (52.13,112.55) -- (70.89,133.4) -- (53.53,111.19) -- (72.29,132.04) ;
\draw [color=black, opacity=0.4]  (44.62,92.96) -- (52.35,65.49) -- (46.5,93.39) -- (54.23,65.92) -- (48.38,93.81) -- (56.1,66.34) -- (50.26,94.24) ;
\draw  [color=black, opacity=0.4] (65.01,47.93) -- (90.32,36.68) -- (65.89,49.71) -- (91.2,38.45) -- (66.77,51.49) -- (92.08,40.23) -- (67.64,53.26) ;
\draw  [line width=1pt, fill=black, fill opacity=0.2] (201.64,44.51) .. controls (207.58,39.46) and (216.34,40.37) .. (221.2,46.54) .. controls (226.07,52.72) and (225.19,61.82) .. (219.25,66.88) .. controls (213.31,71.93) and (204.55,71.03) .. (199.69,64.85) .. controls (194.82,58.67) and (195.7,49.57) .. (201.64,44.51) -- cycle ;
\draw [line width=1pt, fill=black, fill opacity=0.2]  (286.87,44.38) .. controls (292.81,39.32) and (301.57,40.23) .. (306.43,46.41) .. controls (311.29,52.58) and (310.42,61.69) .. (304.48,66.74) .. controls (298.54,71.8) and (289.78,70.89) .. (284.91,64.71) .. controls (280.05,58.54) and (280.92,49.43) .. (286.87,44.38) -- cycle ;
\draw [line width=1pt, fill=black, fill opacity=0.2]  (193.83,95.81) .. controls (197.99,92.27) and (204.12,92.91) .. (207.52,97.23) .. controls (210.93,101.55) and (210.31,107.93) .. (206.15,111.47) .. controls (202,115.01) and (195.86,114.37) .. (192.46,110.05) .. controls (189.05,105.72) and (189.67,99.35) .. (193.83,95.81) -- cycle ;
\draw [line width=1pt, fill=black, fill opacity=0.2]  (300.1,95.64) .. controls (304.26,92.1) and (310.39,92.74) .. (313.8,97.06) .. controls (317.2,101.39) and (316.59,107.76) .. (312.43,111.3) .. controls (308.27,114.84) and (302.14,114.2) .. (298.74,109.88) .. controls (295.33,105.56) and (295.94,99.18) .. (300.1,95.64) -- cycle ;
\draw  [color=black, opacity=0.4]  (253.03,34.3) -- (286.87,44.38) ;
\draw  [color=black, opacity=0.4]  (253.03,34.3) -- (284.5,47.85) ;
\draw  [color=black, opacity=0.4]  (253.03,34.3) -- (282.28,51.46) ;
\draw  [color=black, opacity=0.4]  (276.83,141.95) -- (302.3,112.59) ;
\draw  [color=black, opacity=0.4]  (276.83,141.95) -- (297.57,108.11) ;
\draw [color=black, opacity=0.4]   (276.83,141.95) -- (299.85,110.46) ;
\draw  [color=black, opacity=0.4]  (229.54,142.03) -- (203.3,113.17) ;
\draw [color=black, opacity=0.4]   (229.54,142.03) -- (206.16,111.47) ;
\draw  [color=black, opacity=0.4]  (229.54,142.03) -- (208.31,108.54) ;
\draw [color=black, opacity=0.4]   (253.03,34.3) -- (217.21,43.22) ;
\draw  [color=black, opacity=0.4]  (253.03,34.3) -- (221.2,46.54) ;
\draw  [color=black, opacity=0.4]  (253.03,34.3) -- (223.32,50.74) ;
\draw  [color=black, opacity=0.4]  (276.83,141.95) -- (229.54,142.03) ;
\draw [fill=uuuuuu] (253.03,34.3) circle (1.5pt);
\draw [fill=uuuuuu] (276.83,141.95) circle (1.5pt);
\draw [fill=uuuuuu] (229.54,142.03) circle (1.5pt);
%
\draw  [color=black, opacity=0.4] (197.93,94.47) -- (205.25,66.89) -- (199.82,94.87) -- (207.13,67.28) -- (201.7,95.27) -- (209.02,67.68) -- (203.58,95.66) ;
\draw [color=black, opacity=0.4]  (302.88,96.22) -- (297.07,68.25) -- (304.73,95.66) -- (298.91,67.69) -- (306.57,95.11) -- (300.76,67.14) -- (308.42,94.55) ;
\draw [line width=1pt, fill=black, fill opacity=0.2]  (415.28,36.59) .. controls (415.28,28.53) and (421.53,22) .. (429.23,22) .. controls (436.93,22) and (443.17,28.53) .. (443.17,36.59) .. controls (443.17,44.65) and (436.93,51.19) .. (429.23,51.19) .. controls (421.53,51.19) and (415.28,44.65) .. (415.28,36.59) -- cycle ;
\draw [line width=1pt, fill=black, fill opacity=0.2]  (448.5,49.08) .. controls (448.5,41.02) and (454.74,34.49) .. (462.44,34.49) .. controls (470.14,34.49) and (476.38,41.02) .. (476.38,49.08) .. controls (476.38,57.14) and (470.14,63.68) .. (462.44,63.68) .. controls (454.74,63.68) and (448.5,57.14) .. (448.5,49.08) -- cycle ;
\draw [line width=1pt, fill=black, fill opacity=0.2]  (433.35,141.25) .. controls (433.35,133.19) and (439.59,126.66) .. (447.29,126.66) .. controls (454.99,126.66) and (461.23,133.19) .. (461.23,141.25) .. controls (461.23,149.31) and (454.99,155.85) .. (447.29,155.85) .. controls (439.59,155.85) and (433.35,149.31) .. (433.35,141.25) -- cycle ;
\draw [line width=1pt, fill=black, fill opacity=0.2]  (398.06,141.41) .. controls (398.06,133.35) and (404.3,126.81) .. (412,126.81) .. controls (419.7,126.81) and (425.94,133.35) .. (425.94,141.41) .. controls (425.94,149.47) and (419.7,156) .. (412,156) .. controls (404.3,156) and (398.06,149.47) .. (398.06,141.41) -- cycle ;
\draw [line width=1pt, fill=black, fill opacity=0.2]  (382.17,49.37) .. controls (382.17,41.31) and (388.42,34.78) .. (396.12,34.78) .. controls (403.82,34.78) and (410.06,41.31) .. (410.06,49.37) .. controls (410.06,57.43) and (403.82,63.97) .. (396.12,63.97) .. controls (388.42,63.97) and (382.17,57.43) .. (382.17,49.37) -- cycle ;
\draw [line width=1pt, fill=black, fill opacity=0.2]  (470.45,80.99) .. controls (470.45,75.35) and (474.82,70.78) .. (480.21,70.78) .. controls (485.6,70.78) and (489.97,75.35) .. (489.97,80.99) .. controls (489.97,86.63) and (485.6,91.21) .. (480.21,91.21) .. controls (474.82,91.21) and (470.45,86.63) .. (470.45,80.99) -- cycle ;
\draw [line width=1pt, fill=black, fill opacity=0.2]  (464.47,117.39) .. controls (464.47,111.75) and (468.84,107.18) .. (474.23,107.18) .. controls (479.62,107.18) and (483.99,111.75) .. (483.99,117.39) .. controls (483.99,123.03) and (479.62,127.61) .. (474.23,127.61) .. controls (468.84,127.61) and (464.47,123.03) .. (464.47,117.39) -- cycle ;
\draw [line width=1pt, fill=black, fill opacity=0.2]  (375.11,117.78) .. controls (375.11,112.14) and (379.48,107.57) .. (384.87,107.57) .. controls (390.26,107.57) and (394.63,112.14) .. (394.63,117.78) .. controls (394.63,123.43) and (390.26,128) .. (384.87,128) .. controls (379.48,128) and (375.11,123.43) .. (375.11,117.78) -- cycle ;
\draw [line width=1pt, fill=black, fill opacity=0.2]  (368.84,81.44) .. controls (368.84,75.79) and (373.21,71.22) .. (378.6,71.22) .. controls (383.99,71.22) and (388.36,75.79) .. (388.36,81.44) .. controls (388.36,87.08) and (383.99,91.65) .. (378.6,91.65) .. controls (373.21,91.65) and (368.84,87.08) .. (368.84,81.44) -- cycle ;
\draw  [color=black, opacity=0.4] (403.3,42.67) -- (418.84,36.84) -- (404.06,44.31) -- (419.59,38.47) -- (404.81,45.94) -- (420.34,40.11) -- (405.56,47.57) ;
\draw  [color=black, opacity=0.4] (440.02,35.93) -- (454.22,44.72) -- (439.3,37.58) -- (453.5,46.37) -- (438.58,39.22) -- (452.78,48.02) -- (437.86,40.87) ;
\draw [color=black, opacity=0.4]  (437.97,145.95) -- (421.5,144.76) -- (437.9,144.14) -- (421.44,142.95) -- (437.84,142.33) -- (421.38,141.14) -- (437.78,140.52) ;
\draw [color=black, opacity=0.4]  (469.43,56.75) -- (479.63,75.19) -- (467.9,57.88) -- (478.1,76.32) -- (466.37,59.02) -- (476.57,77.46) -- (464.84,60.15) ;
\draw [color=black, opacity=0.4]  (481.53,89.09) -- (477.4,109.95) -- (479.67,88.87) -- (475.53,109.73) -- (477.8,88.65) -- (473.67,109.51) -- (475.94,88.43) ;
\draw [color=black, opacity=0.4]  (473.14,124.2) -- (456.63,136.66) -- (471.91,122.72) -- (455.4,135.18) -- (470.68,121.24) -- (454.17,133.7) -- (469.45,119.76) ;
\draw [color=black, opacity=0.4]  (402.28,137.86) -- (387.54,123.16) -- (403.44,136.32) -- (388.7,121.62) -- (404.6,134.78) -- (389.86,120.09) -- (405.77,133.24) ;
\draw [color=black, opacity=0.4]  (380.71,110.01) -- (378.01,88.89) -- (382.54,109.57) -- (379.84,88.45) -- (384.37,109.13) -- (381.66,88.01) -- (386.19,108.69) ;
\draw [color=black, opacity=0.4]  (378.35,74.65) -- (389.18,56.61) -- (380.02,75.54) -- (390.85,57.5) -- (381.69,76.42) -- (392.52,58.38) -- (383.36,77.31) ;
\draw  [line width=1pt, fill=black, fill opacity=0.2] (558.42,36.61) .. controls (558.42,30.97) and (562.79,26.39) .. (568.18,26.39) .. controls (573.57,26.39) and (577.94,30.97) .. (577.94,36.61) .. controls (577.94,42.25) and (573.57,46.82) .. (568.18,46.82) .. controls (562.79,46.82) and (558.42,42.25) .. (558.42,36.61) -- cycle ;
\draw [line width=1pt, fill=black, fill opacity=0.2]  (538.99,142.45) .. controls (538.99,136.8) and (543.36,132.23) .. (548.75,132.23) .. controls (554.14,132.23) and (558.51,136.8) .. (558.51,142.45) .. controls (558.51,148.09) and (554.14,152.66) .. (548.75,152.66) .. controls (543.36,152.66) and (538.99,148.09) .. (538.99,142.45) -- cycle ;
\draw [line width=1pt, fill=black, fill opacity=0.2]  (574.74,143.02) .. controls (574.74,137.38) and (579.11,132.81) .. (584.5,132.81) .. controls (589.9,132.81) and (594.27,137.38) .. (594.27,143.02) .. controls (594.27,148.66) and (589.9,153.24) .. (584.5,153.24) .. controls (579.11,153.24) and (574.74,148.66) .. (574.74,143.02) -- cycle ;
\draw [line width=1pt, fill=black, fill opacity=0.2]  (520.45,48.87) .. controls (520.45,40.81) and (526.69,34.28) .. (534.39,34.28) .. controls (542.09,34.28) and (548.34,40.81) .. (548.34,48.87) .. controls (548.34,56.93) and (542.09,63.46) .. (534.39,63.46) .. controls (526.69,63.46) and (520.45,56.93) .. (520.45,48.87) -- cycle ;
\draw [line width=1pt, fill=black, fill opacity=0.2]  (587.65,49.95) .. controls (587.65,41.89) and (593.9,35.36) .. (601.6,35.36) .. controls (609.3,35.36) and (615.54,41.89) .. (615.54,49.95) .. controls (615.54,58.01) and (609.3,64.54) .. (601.6,64.54) .. controls (593.9,64.54) and (587.65,58.01) .. (587.65,49.95) -- cycle ;
\draw  [color=black, opacity=0.4] (576.83,35.44) -- (591.03,44.24) -- (576.11,37.09) -- (590.31,45.88) -- (575.39,38.74) -- (589.59,47.53) -- (574.67,40.38) ;
\draw  [color=black, opacity=0.4] (561.5,40.72) -- (545.69,45.67) -- (560.83,39.05) -- (545.02,44) -- (560.16,37.37) -- (544.35,42.33) -- (559.5,35.7) ;
\draw  [color=black, opacity=0.4]  (516.04,80.99) -- (521.71,117.95) ;
\draw [color=black, opacity=0.4]   (619,82.65) -- (612.25,119.4) ;
\draw [color=black, opacity=0.4]   (521.3,53.71) -- (516.04,80.99) ;
\draw  [color=black, opacity=0.4]  (524.36,58.38) -- (516.04,80.99) ;
\draw [color=black, opacity=0.4]   (527.43,61.59) -- (516.04,80.99) ;
\draw  [color=black, opacity=0.4]  (538.99,142.45) -- (521.71,117.95) ;
\draw  [color=black, opacity=0.4]  (539.98,138.35) -- (521.71,117.95) ;
\draw  [color=black, opacity=0.4]  (541.65,134.85) -- (521.71,117.95) ;
\draw  [color=black, opacity=0.4]  (591.29,135.44) -- (612.25,119.4) ;
\draw [color=black, opacity=0.4]   (593.24,138.06) -- (612.25,119.4) ;
\draw  [color=black, opacity=0.4]  (594.36,141.56) -- (612.25,119.4) ;
\draw [color=black, opacity=0.4]   (610.53,61.01) -- (619,82.65) ;
\draw [color=black, opacity=0.4]   (613.32,57.51) -- (619,82.65) ;
\draw  [color=black, opacity=0.4]  (607.75,62.76) -- (619,82.65) ;
\draw [color=black, opacity=0.4]  (580.14,147.4) -- (555.72,146.83) -- (580.09,145.89) -- (555.67,145.32) -- (580.03,144.39) -- (555.62,143.81) -- (579.98,142.88) ;
\draw [fill=uuuuuu] (516.04,80.99) circle (1.5pt);
\draw [fill=uuuuuu] (521.71,117.95) circle (1.5pt);
\draw [fill=uuuuuu] (619,82.65) circle (1.5pt);
\draw [fill=uuuuuu] (612.25,119.4) circle (1.5pt);
\end{tikzpicture}
\caption{Extremal constructions for Theorem~\ref{THM:max-deg-AES-odd-cycle} when $k = 2$ and $k = 3$.} 
\label{Fig:C7C9}
\end{figure}

Recall that $C_{2k+3}$ is the graph with vertex set $[2k+3]$ and edge set 
\begin{align*}
    \big\{ \{1,2\},~\{2,3\},~\ldots,~\{2k+2, 2k+3\},~\{2k+3,1\} \big\}. 
\end{align*}

The extremal construction for Theorem~\ref{THM:max-deg-AES-odd-cycle} in the case $\Delta(G) < \frac{2n}{k+2}$ is given by the blowup $C_{2k+3}[x_1, \ldots, x_{2k+3}]$, where $\alpha \coloneqq \frac{k+2}{2}\left(\frac{\Delta(G)}{n}-\frac{2}{2k+3}\right)$ (noting that Theorem~\ref{THM:max-deg-AES-odd-cycle} implicitly requires $\Delta(G) > \frac{2 n}{2k+3}$), and for $i \in [2k+3]$,
\begin{align*}
    x_i 
    \coloneqq 
    \begin{cases}
        \big(\tfrac{1}{2k+3} +\tfrac{\alpha}{k+2} \big)n, & \quad\text{if}\quad i \equiv 1 \mbox{ or } 2 \Mod{4}, \\[0.3em]
        \big(\tfrac{1}{2k+3} - \tfrac{\alpha}{k+1} \big)n, & \quad\text{if}\quad i \equiv 0 \mbox{ or } 3 \Mod{4}.  
    \end{cases}
\end{align*}

Next, we describe the construction for the case $\Delta(G)\ge \frac{2n}{k+2}$. 
Let $\beta \coloneqq \frac{\Delta(G)}{n}-\frac{2}{k+2}$. 
If $k \equiv 0 \Mod{2}$, define 
\begin{align*}
    y_i 
    \coloneqq 
    \begin{cases}
        \big(\frac{1}{k+2} + \frac{\beta}{2} \big)n, & \quad\text{if}\quad i \in \{1,3\}, \\
        1, & \quad\text{if}\quad i = 2, \\
        \big(\tfrac{k}{k(k+2)} -\frac{\beta}{k}\big)n-\frac{2k+1}{k}, & \quad\text{if}\quad i\in [4, 2k+3] \mbox{ and } i \equiv 1 \mbox{ or } 2 \Mod{4}, \\
        1, & \quad\text{if}\quad i\in [4, 2k+3] \mbox{ and } i \equiv 0 \mbox{ or } 3 \Mod{4}. 
    \end{cases}
\end{align*}

If $k \equiv 1 \Mod{2}$, define 
\begin{align*}
    y_i 
    \coloneqq 
    \begin{cases}
        \big(\frac{1}{k+2} + \frac{\beta}{2} \big)n, & \quad\text{if}\quad i \in \{1,3\}, \\
        \big(\tfrac{k}{k(k+2)} -\frac{\beta}{k}\big)n-\frac{2k+1}{k}, & \quad\text{if}\quad i = 2, \\
        1, & \quad\text{if}\quad i\in [4, 2k+3] \mbox{ and } i \equiv 0 \mbox{ or } 1 \Mod{4}, \\
        \big(\tfrac{k}{(k-1)(k+2)} -\frac{\beta}{k-1}\big)n-\frac{2k}{k-1}, & \quad\text{if}\quad i\in [4, 2k+3] \mbox{ and } i \equiv 2 \mbox{ or } 3 \Mod{4}. 
    \end{cases}
\end{align*}

The extremal construction for Theorem~\ref{THM:max-deg-AES-odd-cycle} in the case $\Delta(G)\ge \frac{2n}{k+2}$ is given by the blowup $C_{2k+3}[y_1, \ldots, y_{2k+3}]$.

In Section~\ref{SEC:proof-clique}, we present the proof of Theorem~\ref{THM:max-deg-AES-clique}, and in Section~\ref{SEC:proof-odd-cycle}, we present the proof of Theorem~\ref{THM:max-deg-AES-odd-cycle}.

\section{Proof of Theorem~\ref{THM:max-deg-AES-clique}}\label{SEC:proof-clique}
In this section, we present the proof of Theorem~\ref{THM:max-deg-AES-clique}.
Our proof proceeds by induction on $r$. 
We begin with the base case $r=2$.

\begin{lemma}\label{LEMMA:max-deg-triangle}
    Suppose that $G$ is an $n$-vertex $K_{3}$-free graph satisfying 
    \begin{align*}
        \delta(G) 
        > \min\left\{\frac{n}{2}-\frac{\Delta(G)}{4},~n-\Delta(G)-1\right\}. 
    \end{align*}
    Then $G$ is bipartite. 
\end{lemma}
\begin{proof}[Proof of Lemma~\ref{LEMMA:max-deg-triangle}]
    Let $G$ be an $n$-vertex $K_{3}$-free graph.  
    For convenience, let $V \coloneqq V(G)$, $\Delta \coloneqq \Delta(G)$, and $\delta \coloneqq \delta(G)$. 

    \bigskip 
    
    \textbf{Case 1}: $\delta > n - \Delta -1$. 

    Since both $\delta$ and $\Delta$ are integers, this assumption is equivalent to  $\delta + \Delta \ge n$. 
    Fix a vertex $u \in V$ with $d_{G}(u) = \Delta$, and let $v \in N_{G}(u)$. Since $G$ is $K_{3}$-free, we have $N_{G}(u) \cap N_{G}(v) = \emptyset$. 
    Therefore, 
    \begin{align*}
        |N_{G}(u)| + |N_{G}(v)| 
        \ge \Delta + \delta 
        = n. 
    \end{align*}
    Thus, $N_{G}(u) \cup N_{G}(v) = V$, and this union is a partition. 
    By the $K_{3}$-freeness of $G$, both $N_{G}(u)$ and $N_{G}(v)$ are independent sets. 
    Thus, $G$ is bipartite with parts $N_{G}(u)$ and $N_{G}(v)$.

    \bigskip 

    \textbf{Case 2}: $\delta > \frac{n}{2}-\frac{\Delta}{4}$. 

    Similar to the proof of Case 1. 
    Let $\{u,v\} \in G$ be an edge with $d_{G}(u) = \Delta$. 
    Let $V_1$ and $V_2$ be two disjoint independent sets in $G$ such that $N_{G}(v) \subseteq V_1$ and $N_{G}(u) \subseteq V_2$, and such that $|V_1| + |V_2|$ is maximized. 
    Let $T \coloneqq V \setminus (V_1 \cup V_2)$. 
    Since $N_{G}(u) \cap N_{G}(v) = \emptyset$, we have 
    \begin{align}\label{equ:K3-T-size}
        |T|
        \le n - |N_{G}(u)| - |N_{G}(v)|
        \le n - \Delta - \delta. 
    \end{align}
    We may assume $T \neq \emptyset$, since otherwise, $G$ is a bipartite graph with parts $V_1$ and $V_2$. 
    Fix a vertex $w\in T$. 
    By the maximality of $|V_1| + |V_2|$, the vertex $w$ must have neighbors in both $V_1$ and $V_2$ (otherwise, we can move $w$ to $V_1$ or $V_2$). 
    Let $N_1 \coloneqq N_{G}(w) \cap V_1$ and $N_2 \coloneqq N_{G}(w) \cap V_2$, and choose $v_1 \in N_1$ and $v_2 \in N_2$. 
    Since $V_1$ and $V_2$ are independent, we have 
    \begin{align}\label{equ:K3-v1v2-intersection}
        N_{G}(v_1) \cap V_1
        = N_{G}(v_2) \cap V_2
        = \emptyset. 
    \end{align}
    Since $G$ is $K_3$-free, we also have $N_{G}(v_1) \cap N_{G}(w) = N_{G}(v_2) \cap N_{G}(w)= \emptyset$. 
    It follows that 
    \begin{align*}
        |N_{G}(v_1) \cap T| + |N_{G}(v_2) \cap T| + |N_{G}(w) \cap T|
        & \le 2|T|, \\
        |N_{G}(v_1) \cap V_2| & \le |V_2| - |N_2|, \\
        |N_{G}(v_2) \cap V_1| & \le |V_1| - |N_1|.
    \end{align*}
    Combining these inequalities with~\eqref{equ:K3-v1v2-intersection}, we obtain 
    \begin{align*}
        d_{G}(v_1) + d_{G}(v_2) + d_{G}(w)
        \le |V_2| - |N_2| + |V_1| - |N_1| + |N_1| + |N_2| + 2|T|
        = n + |T|. 
    \end{align*}
    Together with~\eqref{equ:K3-T-size}, this yields
    \begin{align*}
        3\delta 
        \le n + |T|
        \le n + n - \Delta - \delta, 
    \end{align*}
    which simplifies to $\delta \le \frac{n}{2} - \frac{\Delta}{4}$, contradicting our assumption.
\end{proof}

We are now ready to present the proof of Theorem~\ref{THM:max-deg-AES-clique}. 

\begin{proof}[Proof of Theorem~\ref{THM:max-deg-AES-clique}]
    We prove this theorem by induction on $r$. 
    The base case $r=2$ is established in Lemma~\ref{LEMMA:max-deg-triangle}.
    Thus we may assume that $r\ge 3$. 
    
    Let $G$ be an $n$-vertex $K_{r+1}$-free graph.  
    For convenience, let $V \coloneqq V(G)$, $\Delta \coloneqq \Delta(G)$, and $\delta \coloneqq \delta(G)$. 

    \bigskip

    \textbf{Case 1}: $\delta > n - \frac{\Delta+1}{r-1}$. 

    Since both $\delta$ and $\Delta$ are integers, the assumption is equivalent to $(r-1)\delta + \Delta \ge (r-1)n$ (and thus also to $\delta \ge n - \frac{\Delta}{r-1}$). 
    Thus, by adding edges to $G$ while preserving the condition that $G$ remains $K_{r+1}$-free, we may assume that $G$ is maximal, that is, the addition of any new edge to $G$ creates a copy of $K_{r+1}$.  

    Let $u \in V$ be a vertex of maximum degree in $G$. 
    Let $X \coloneqq N_{G}(u)$ and let $H \coloneqq G[X]$. 
    Since $G$ is $K_{r+1}$-free, the graph $H$ is $K_{r}$-free. 

    \begin{claim}\label{CLAIM:Kr-H-induction-case1}
        The graph $H$ is $(r-1)$-partite.
    \end{claim}
    \begin{proof}[Proof of Claim~\ref{CLAIM:Kr-H-induction-case1}]
        Fix a vertex $v \in X$ of minimum degree in $H$. 
        It follows from the assumption on $\delta$ that  
        \begin{align*}
            \frac{\delta(H)}{|X|}
            \ge \frac{d_{G}(v) - \left( n-|X| \right)}{|X|}
            \ge \frac{\delta - (n-\Delta)}{\Delta}
            \ge \frac{1}{\Delta}\left( n - \frac{\Delta}{r-1} - (n-\Delta)\right)
            = \frac{r-2}{r-1}. 
        \end{align*}
        Rearranging yields $(r-1)\delta(H) \ge (r-2)|X|$, and thus 
        \begin{align*}
            (r-2)\delta(H) + \Delta(H)
            \ge (r-1)\delta(H) \ge (r-2)|X|. 
        \end{align*}
        By the inductive hypothesis, $H$ is $(r-1)$-partite. 
    \end{proof}
    
    Let $V_2 \cup \cdots \cup V_{r} = X$ be an $(r-1)$-partition such that $H[V_i] = \emptyset$ for $i \in [2,r]$. 
    We may assume that there exists a vertex $w\in V\setminus (X \cup \{u\})$, since otherwise $G$ is already $r$-partite. 
    Note that $\{u,w\} \not\in G$. 
    Thus, by the maximality of $G$, the induced subgraph of $G$ on $N_{G}(u) \cap N_{G}(w) \subseteq N_{G}(u)$ contains a copy of $K_{r-1}$.  
    It follows that there exist vertices $v_i \in V_i$ for $i \in [2,r]$ such that $\{v_2, \ldots, v_{r}\}$ induces a copy of $K_{r-1}$ in $G$. 
    
    Let $V_1 \coloneqq \bigcap_{i=2}^{r} N_{G}(v_i) \subseteq V\setminus X$, noting that $V_1$ is independent in $G$. 
    Using the assumption on $\delta$ and the Inclusion-Exclusion Principle, we obtain 
    \begin{align*}
        |V_1|
        \ge (r-1)\delta - (r-2)n
        \ge (r-1)\left(n - \frac{\Delta}{r-1}\right) - (r-2)n
        \ge n - \Delta
        = n - |X|. 
    \end{align*}
    Thus, $V_1 = V\setminus X$, and hence, $V_1 \cup \cdots \cup V_r$ is a partition of $V$. 
    Consequently, $G$ is $r$-partite with parts $V_1, \ldots, V_{r}$.

    \bigskip

    \textbf{Case 2}: $\delta > \frac{3r-4}{3r-2}n - \frac{\Delta}{3r-2}$.

    Note that the assumption is equivalent to that $(3r-2)\delta + \Delta > (3r-4)n$. As in Case 1, we may assume that $G$ is maximal. 

    Let $u \in V$ be a vertex of maximum degree in $G$. 
    Let $X \coloneqq N_{G}(u)$ and let $H \coloneqq G[X]$. 
    Since $G$ is $K_{r+1}$-free, the graph $H$ is $K_{r}$-free. 

    \begin{claim}\label{CLAIM:Kr-H-induction-case2}
        The graph $H$ is $(r-1)$-partite.
    \end{claim}
    \begin{proof}[Proof of Claim~\ref{CLAIM:Kr-H-induction-case2}]
        Fix a vertex $v \in X$ of minimum degree in $H$. 
        It follows from the assumption on $\delta$ that 
        \begin{align*}
            \frac{\delta(H)}{|X|}
            \ge \frac{d_{G}(v) - \left( n-|X| \right)}{|X|}
            \ge \frac{\delta - (n-\Delta)}{\Delta}
            & > \frac{1}{\Delta}\left( \frac{3r-4}{3r-2}n - \frac{\Delta}{r-1} - (n-\Delta)\right) \\
            & = \frac{3r-3}{3r-2} - \frac{2n}{(3r-2)\Delta}. 
        \end{align*}
        Since $(3r-1)\Delta \ge (3r-2)\delta + \Delta > (3r-4)n$, we have $\Delta > \frac{3r-4}{3r-1}n$. 
        Substituting this bound into the previous inequality gives 
        \begin{align*}
            \frac{\delta(H)}{|X|}
            > \frac{3r-3}{3r-2} - \frac{2}{3r-2} \cdot \frac{3r-1}{3r-4}
            = \frac{3r-7}{3r-4}. 
        \end{align*}
        Therefore, 
        \begin{align*}
            (3r-5) \delta(H) + \Delta(H)
            \ge (3r-4)\delta(H)
            > (3r-7) |X|. 
        \end{align*}
        So it follows from the inductive hypothesis that $H$ is $(r-1)$-partite. 
    \end{proof}

    Let $V_2 \cup \cdots \cup V_{r} = X$ be an $(r-1)$-partition such that $H[V_i] = \emptyset$ for $i \in [2,r]$. 
    As in Case 1, by the maximality of $G$, there exist vertices $v_i \in V_i$ for $i \in [2,r]$ such that $\{v_2, \ldots, v_{r}\}$ induces a copy of $K_{r-1}$ in $G$. 
    
    Let $V_1$ be a maximum independent set in $G$ subject to the condition that $\bigcap_{i=2}^{r} N_{G}(v_i) \subseteq V_1$. 
    This is well-defined because $\bigcap_{i=2}^{r} N_{G}(v_i)$ is independent in $G$ (by the $K_{r+1}$-freeness of $G$). 
    Observe also that $V_1 \subseteq V\setminus X$. 
    It follows from the assumption on $\delta(G)$ and the Inclusion-Exclusion Principle that 
    \begin{align}\label{equ:Kr-V1-size}
        |V_1|
        \ge \left| \cap_{i=2}^{r} N_{G}(v_i) \right|
        \ge (r-1)\delta - (r-2)n
        & > (r-1)\left(\frac{3r-4}{3r-2}n - \frac{\Delta}{3r-2}\right) - (r-2)n \notag \\
        & = \frac{r}{3r-2} n - \frac{r-1}{3r-2}\Delta. 
    \end{align}
    Let $Z \coloneqq V \setminus (V_1 \cup \cdots \cup V_{r})$. 
    Using~\eqref{equ:Kr-V1-size}, we obtain 
    \begin{align}\label{equ:Kr-Z-size}
        |Z|
        = n - \left(|V_1| + \cdots + |V_{r}|\right)
        = n - |V_1| - |X|
        & < n - \left(\frac{r}{3r-2} n - \frac{r-1}{3r-2}\Delta\right) - \Delta \notag \\
        & = \frac{2r-2}{3r-2} n - \frac{2r-1}{3r-2}\Delta. 
    \end{align}
    We may assume that $Z \neq \emptyset$, since otherwise $G$ is $r$-partite with parts $V_1, \ldots, V_{r}$. 
    Fix a vertex $z \in Z$.
    It follows from the maximality of $V_1$ that $N_{G}(z) \cap V_1 \neq \emptyset$. 

    Fix $u_1 \in N_{G}(z) \cap V_1$. Since $X = N_{G}(u)$ and $z \not\in X$, we have $u_1 \neq u$. 
    Let $S \subseteq X$ be a largest subset such that $\{z,u_1\} \cup S$ induces a complete subgraph in $G$, noting from the $K_{r+1}$-freeness of $G$ that $|S| \le r-2$. 
    
    Let $j \coloneqq |S|+1 \le r-1$, and assume that $S = \{u_2, \ldots, u_j\}$. 
    For each $i \in [2,r]$, since $G[V_i] = \emptyset$, we have $|S\cap V_i| \le 1$. 
    By relabeling the sets $V_2, \ldots, V_r$ if necessary, we may assume that $u_i \in V_i$ for $i \in [2,j]$. 
    
    \begin{claim}\label{CLAIM:Kr-max-clique}
        We have $j = r-1$.
    \end{claim}
    \begin{proof}[Proof of Claim~\ref{CLAIM:Kr-max-clique}]
        Let $N \coloneqq N_{G}(z) \cap N_{G}(u_1) \cap \cdots \cap N_{G}(u_{j})$. 
        Suppose to the contrary that $j \le r-2$. 
        Then it follows from the assumption on $\delta(G)$ and the Inclusion-Exclusion Principle that 
        \begin{align*}
            |N|
            \ge (j+1) \delta - j n
            \ge (r-2 +1) \delta - (r-2) n
            & > (r-1) \left(\frac{3r-4}{3r-2} n - \frac{\Delta}{3r-2}\right) - (r-2)n \\
            & = \frac{r}{3r-2} n - \frac{r-1}{3r-2} \Delta.  
        \end{align*}
        It follows from the maximality of $S$ that $N \subseteq Z$, and hence, $|N| \le |Z|$. 
        Combining this with~\eqref{equ:Kr-Z-size}, we obtain 
        \begin{align*}
            \frac{r}{3r-2} n - \frac{r-1}{3r-2} \Delta
            < \frac{2r-2}{3r-2} n - \frac{2r-1}{3r-2}\Delta, 
        \end{align*}
        which implies that 
        \begin{align*}
            \Delta 
            < \frac{r-2}{r} n 
            < \frac{3r-4}{3r-1}n, 
        \end{align*}
        a contradiction to our assumption. 
        Thus, $j = r-1$, proving the claim. 
    \end{proof}

    Claim~\ref{CLAIM:Kr-max-clique} shows that $\{z, u_1, \ldots, u_{r-1}\}$ induces a copy of $K_{r}$ in $G$. 
    Recall that $X = N_{G}(u)$ and $z \not\in X$.
    By the maximality of $G$, there exists an $(r-1)$-set $W = \{w_2, \ldots, w_{r}\} \subseteq N_{G}(u) \cap N_{G}(z) \subseteq X$ such that $W$ induces a copy of $K_{r-1}$ in $G$. 
    Since $G[X]$ is $(r-1)$-partite, $|W \cap V_i| = 1$ for $i \in [2,r]$. 
    By relabeling vertices in $W$ if necessary, we may assume that $w_i \in V_i$ for $i \in [2,r]$. 

    Let $N_1 \coloneqq N_{G}(u_1) \cap \cdots \cap N_{G}(u_{r-1})$ and $N_2 \coloneqq N_{G}(w_2) \cap \cdots \cap N_{G}(w_r)$, noting that for $i \in \{1,2\}$, 
    \begin{align}\label{equ:Kr-Ni-size}
        |N_i|
        \ge (r-1)\delta - (r-2)n
        & > (r-1)\left(\frac{3r-4}{3r-1} n - \frac{\Delta}{3r-2}\right) - (r-2)n \notag \\
        & = \frac{r}{3r-2} n - \frac{r-1}{3r-2} \Delta. 
    \end{align}
    Observe that $N_1 \subseteq Z\cup V_{r}$ and $N_{2} \subseteq Z \cup V_1$, so $N_1 \cap N_2 \subseteq Z$. 
    Also, since $G$ is $K_{r+1}$-free, we have $N_{G}(z) \cap N_1 = N_{G}(z) \cap N_2 = \emptyset$. 
    Combining these with~\eqref{equ:Kr-Z-size} and~\eqref{equ:Kr-Ni-size}, we obtain 
    \begin{align*}
        |N_1 \cup N_2 \cup N_{G}(z)|
        & = |N_1| + |N_2| + |N_{G}(z)| - |N_1 \cap N_2| \\
        & \ge |N_1| + |N_2| + |N_{G}(z)| - |Z| \\
        & > 2\left(\frac{r}{3r-2} n - \frac{r-1}{3r-2} \Delta\right) + \delta - \left(\frac{2r-2}{3r-2} n - \frac{2r-1}{3r-2}\Delta\right) \\
        & > \frac{2}{3r-2}n + \frac{\Delta}{3r-2} + \frac{3r-4}{3r-2}n - \frac{\Delta}{3r-2}
        = n, 
    \end{align*}
    a contradiction. 
    This completes the proof of Theorem~\ref{THM:max-deg-AES-clique}. 
\end{proof}

\section{Proof of Theorem~\ref{THM:max-deg-AES-odd-cycle}}\label{SEC:proof-odd-cycle}
In this section, we present the proof of Section~\ref{THM:max-deg-AES-odd-cycle}. 
The following simple fact will be used multiple times. 
For convenience, we use $g_{\mathrm{odd}}(G)$ to denote the length of the shortest odd cycle in $G$. 

\begin{fact}\label{FACT:odd-girth}
    Let $G$ be a graph with $g_{\mathrm{odd}}(G) = 2m+1$ and let $C = v_0 v_1 \cdots v_{2m} \subseteq G$ be a cycle of length $2m+1$ in $G$. 
    For every vertex $u \in V(G)$, either $|N_{G}(u) \cap V(C)| \le 1$ or there exists $i \in [2m+1]$ such that $N_{G}(u) \cap V(C) = \{v_{i}, v_{i+2}\}$, where the indices are taken modulo $2m+1$. 
\end{fact}

\begin{proof}[Proof of Theorem~\ref{THM:max-deg-AES-odd-cycle}]
    Let $G$ be an $n$-vertex $C_{\le 2k+1}$-free graph satisfying the assumptions of the theorem. 
    We may assume that $k \ge 2$, since the case $k=1$ is already established in Lemma~\ref{LEMMA:max-deg-triangle}. 
    For convenience, let $V \coloneqq V(G)$, $\Delta \coloneqq \Delta(G)$, and $\delta \coloneqq \delta(G)$. 
    Suppose that $g_{\mathrm{odd}}(G) = 2m+1$ for some integer $m \ge k+1$. 

    \bigskip

    \textbf{Case 1}: There exists a $(2m+1)$-cycle $C = v_0 v_1 \cdots v_{2m} \subseteq G$ such that $d_{G}(v_0) = \Delta$.  

    \begin{figure}[htbp]
    \centering
    \tikzset{every picture/.style={line width=1pt}} 
    \begin{tikzpicture}[x=0.75pt,y=0.75pt,yscale=-0.7,xscale=0.7]
    \draw   (280.88,131.58) -- (272.15,183.49) -- (232.08,217.65) -- (179.44,218.06) -- (138.84,184.54) -- (129.3,132.77) -- (155.26,86.97) -- (204.59,68.57) -- (254.2,86.19) -- cycle ;
    \draw [color=magenta, fill=magenta] (204.59,68.57) circle (2pt);
    \draw [color=magenta, fill=magenta] (254.2,86.19) circle (2pt);
    \draw [fill=uuuuuu] (280.88,131.58) circle (2pt);
    \draw [fill=uuuuuu] (272.15,183.49) circle (2pt);
    \draw [color=magenta, fill=magenta] (232.08,217.65) circle (2pt);
    \draw [color=magenta, fill=magenta] (179.44,218.06) circle (2pt);
    \draw [fill=uuuuuu] (138.84,184.54) circle (2pt);
    \draw [fill=uuuuuu] (129.3,132.77) circle (2pt);
    \draw [fill=uuuuuu] (155.26,86.97) circle (2pt);
    %
    \draw   (546.9,153.46) -- (528.69,194.14) -- (491.38,218.52) -- (446.81,218.86) -- (409.13,195.05) -- (390.31,154.64) -- (396.32,110.48) -- (425.25,76.57) -- (467.92,63.69) -- (510.78,75.93) -- (540.22,109.39) -- cycle ;
    \draw [color=magenta, fill=magenta] (467.92,63.69) circle (2pt);
    \draw [color=magenta, fill=magenta] (510.78,75.93) circle (2pt);
    \draw [fill=uuuuuu] (540.22,109.39) circle (2pt);
    \draw [fill=uuuuuu] (546.9,153.46) circle (2pt);
    \draw [color=magenta, fill=magenta] (528.69,194.14) circle (2pt);
    \draw [color=magenta, fill=magenta] (491.38,218.52) circle (2pt);
    \draw [fill=uuuuuu] (446.81,218.86) circle (2pt);
    \draw [fill=uuuuuu] (409.13,195.05) circle (2pt);
    \draw [color=magenta, fill=magenta] (390.31,154.64) circle (2pt);
    \draw [fill=uuuuuu] (396.32,110.48) circle (2pt);
    \draw [fill=uuuuuu] (425.25,76.57) circle (2pt);
    %
    \draw (195,49) node [anchor=north west][inner sep=0.75pt]   [align=left] {$v_{0}$};
    \draw (252,67) node [anchor=north west][inner sep=0.75pt]   [align=left] {$v_{1}$};
    \draw (285.33,123.33) node [anchor=north west][inner sep=0.75pt]   [align=left] {$v_{2}$};
    \draw (277.33,177) node [anchor=north west][inner sep=0.75pt]   [align=left] {$v_{3}$};
    \draw (227,224) node [anchor=north west][inner sep=0.75pt]   [align=left] {$v_{4}$};
    \draw (171,224) node [anchor=north west][inner sep=0.75pt]   [align=left] {$v_{5}$};
    \draw (113,180) node [anchor=north west][inner sep=0.75pt]   [align=left] {$v_{6}$};
    \draw (103,123) node [anchor=north west][inner sep=0.75pt]   [align=left] {$v_{7}$};
    \draw (130,70) node [anchor=north west][inner sep=0.75pt]   [align=left] {$v_{8}$};
    %
    %
    \draw (458,43) node [anchor=north west][inner sep=0.75pt]   [align=left] {$v_{0}$};
    \draw (508,57) node [anchor=north west][inner sep=0.75pt]   [align=left] {$v_{1}$};
    \draw (544,98) node [anchor=north west][inner sep=0.75pt]   [align=left] {$v_{2}$};
    \draw (552,146) node [anchor=north west][inner sep=0.75pt]   [align=left] {$v_{3}$};
    \draw (534,197) node [anchor=north west][inner sep=0.75pt]   [align=left] {$v_{4}$};
    \draw (483,225) node [anchor=north west][inner sep=0.75pt]   [align=left] {$v_{5}$};
    \draw (438,225) node [anchor=north west][inner sep=0.75pt]   [align=left] {$v_{6}$};
    \draw (381, 192) node [anchor=north west][inner sep=0.75pt]   [align=left] {$v_{7}$};
    \draw (364, 149) node [anchor=north west][inner sep=0.75pt]   [align=left] {$v_{8}$};
    \draw (367,100) node [anchor=north west][inner sep=0.75pt]   [align=left] {$v_{9}$};
    \draw (392,58) node [anchor=north west][inner sep=0.75pt]   [align=left] {$v_{10}$};
    \end{tikzpicture}
    \caption{The set of selected vertices (in red) for cases $m=4$ and $m=5$.} 
    \label{Fig:C9C11Case1}
    \end{figure}
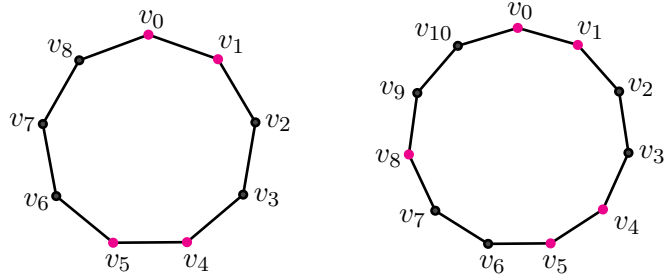

    Fix a $(2m+1)$-cycle $C = v_0 v_1 \cdots v_{2m} \subseteq G$ with $d_{G}(v_0) = \Delta$. For convenience, we slightly abuse notation by writing $C = \{v_0,\ldots,v_{2m}\}$ when no confusion arises. 
    
    Suppose that $\delta > \frac{n}{k+1} - \frac{\Delta}{2k+2}$. 
    Then it follows from Fact~\ref{FACT:odd-girth} that  
    \begin{align*}
        2n 
        \ge \sum_{u \in V} |N_{G}(u) \cap C| 
        & \ge d_{G}(v_0) + \sum_{i=1}^{2m} d_{G}(v_i) \\
        & \ge \Delta + 2m \delta
        > \Delta + 2(k+1)\left(\frac{n}{k+1} - \frac{\Delta}{2k+2}\right)
        = 2n, 
    \end{align*}
    a contradiction.

    Next, suppose that $\delta > \frac{n - \Delta - 1}{k}$. 
    Let (see Figure~\ref{Fig:C9C11Case1})
    \begin{align*}
        S
        \coloneqq
        \begin{cases}
            \left\{ 0,~1,~4,~5,~\ldots,~2m-4,~2m-3 \right\}, & \quad\text{if}\quad m \equiv 0 \Mod{2}, \\
            \left\{ 0,~1,~\ldots,~2m-6,~2m-5,~2m-2 \right\}, & \quad\text{if}\quad m \equiv 1 \Mod{2}. 
        \end{cases}
    \end{align*}
    Observe that for every pair $\{i,j\} \subseteq S$, the distance between $v_i$ and $v_j$ is not $2$, so it follows from Fact~\ref{FACT:odd-girth} that $N_{G}(v_i) \cap N_{G}(v_j) = \emptyset$. 
    Moreover, observe that $v_{2m-1} \not\in \bigcup_{i\in S} N_{G}(v_i)$ for $m \equiv 0 \Mod{2}$, and $v_{2m-2} \not\in \bigcup_{i\in S} N_{G}(v_i)$ for $m \equiv 1 \Mod{2}$. 
    Therefore, 
    \begin{align*}
        n 
        \ge 1 + \sum_{i \in S} d_{G}(v_{i})
        & \ge 1 + \Delta + (|S| - 1) \delta \\
        & = 1 + \Delta + (m-1) \delta 
         > 1 + \Delta + k \cdot \frac{n-\Delta -1}{k}
        = n, 
    \end{align*}
    a contradiction.

    \bigskip

    \textbf{Case 2}: Every $(2m+1)$-cycle $C = v_0 v_i \cdots v_{2m} \subseteq G$ satisfies $\max_{i\in [0,2m]} d_{G}(v_i) < \Delta$.

    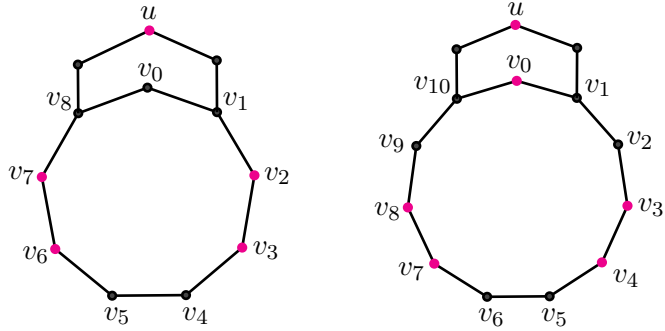
\begin{figure}[htbp]
    \centering
    \tikzset{every picture/.style={line width=1pt}} 
    \begin{tikzpicture}[x=0.75pt,y=0.75pt,yscale=-0.7,xscale=0.7]
    \draw   (280.88,131.58) -- (272.15,183.49) -- (232.08,217.65) -- (179.44,218.06) -- (138.84,184.54) -- (129.3,132.77) -- (155.26,86.97) -- (204.59,68.57) -- (254.2,86.19) -- cycle ;
    \draw [fill=uuuuuu] (204.59,68.57) circle (2pt);
    \draw [fill=uuuuuu] (254.2,86.19) circle (2pt);
    \draw [color=magenta, fill=magenta] (280.88,131.58) circle (2pt);
    \draw [color=magenta, fill=magenta] (272.15,183.49) circle (2pt);
    \draw [fill=uuuuuu] (232.08,217.65) circle (2pt);
    \draw [fill=uuuuuu] (179.44,218.06) circle (2pt);
    \draw [color=magenta, fill=magenta] (138.84,184.54) circle (2pt);
    \draw [color=magenta, fill=magenta] (129.3,132.77) circle (2pt);
    \draw [fill=uuuuuu] (155.26,86.97) circle (2pt);
    %
    \draw    (155,52) -- (155.26,86.97) ;
    \draw    (254,49) -- (254.2,86.19) ;
    \draw    (155,52) -- (206,27.5) ;
    \draw    (206,27.5) -- (254,49) ;
    \draw [color=magenta, fill=magenta] (206,27.5) circle (2pt);
    \draw [fill=uuuuuu] (155,52) circle (2pt);
    \draw [fill=uuuuuu] (254,49) circle (2pt);
    %
    \draw   (546.9,153.46) -- (528.69,194.14) -- (491.38,218.52) -- (446.81,218.86) -- (409.13,195.05) -- (390.31,154.64) -- (396.32,110.48) -- (425.25,76.57) -- (467.92,63.69) -- (510.78,75.93) -- (540.22,109.39) -- cycle ;
    \draw [color=magenta, fill=magenta] (467.92,63.69) circle (2pt);
    \draw [fill=uuuuuu] (510.78,75.93) circle (2pt);
    \draw [fill=uuuuuu] (540.22,109.39) circle (2pt);
    \draw [color=magenta, fill=magenta] (546.9,153.46) circle (2pt);
    \draw [color=magenta, fill=magenta] (528.69,194.14) circle (2pt);
    \draw [fill=uuuuuu] (491.38,218.52) circle (2pt);
    \draw [fill=uuuuuu] (446.81,218.86) circle (2pt);
    \draw [color=magenta, fill=magenta] (409.13,195.05) circle (2pt);
    \draw [color=magenta, fill=magenta] (390.31,154.64) circle (2pt);
    \draw [fill=uuuuuu] (396.32,110.48) circle (2pt);
    \draw [fill=uuuuuu] (425.25,76.57) circle (2pt);\%
    \draw    (425,41) -- (425.25,76.57) ;
    \draw    (511,40) -- (510.78,75.93) ;
    \draw    (425,41) -- (467,23.5) ;
    \draw    (467,23.5) -- (511,40) ;
    \draw [color=magenta, fill=magenta] (467,23.5) circle (2pt);
    \draw [fill=uuuuuu] (425,41) circle (2pt);
    \draw [fill=uuuuuu] (511,40) circle (2pt);
    %
    \draw (197,8) node [anchor=north west][inner sep=0.75pt]   [align=left] {$u$};
    \draw (195,49) node [anchor=north west][inner sep=0.75pt]   [align=left] {$v_{0}$};
    \draw (256,70) node [anchor=north west][inner sep=0.75pt]   [align=left] {$v_{1}$};
    \draw (285.33,123.33) node [anchor=north west][inner sep=0.75pt]   [align=left] {$v_{2}$};
    \draw (277.33,177) node [anchor=north west][inner sep=0.75pt]   [align=left] {$v_{3}$};
    \draw (227,224) node [anchor=north west][inner sep=0.75pt]   [align=left] {$v_{4}$};
    \draw (171,224) node [anchor=north west][inner sep=0.75pt]   [align=left] {$v_{5}$};
    \draw (113,180) node [anchor=north west][inner sep=0.75pt]   [align=left] {$v_{6}$};
    \draw (103,123) node [anchor=north west][inner sep=0.75pt]   [align=left] {$v_{7}$};
    \draw (130,70) node [anchor=north west][inner sep=0.75pt]   [align=left] {$v_{8}$};
    %
    %
    \draw (459,5) node [anchor=north west][inner sep=0.75pt]   [align=left] {$u$};
    \draw (458,43) node [anchor=north west][inner sep=0.75pt]   [align=left] {$v_{0}$};
    \draw (514,60) node [anchor=north west][inner sep=0.75pt]   [align=left] {$v_{1}$};
    \draw (544,98) node [anchor=north west][inner sep=0.75pt]   [align=left] {$v_{2}$};
    \draw (552,146) node [anchor=north west][inner sep=0.75pt]   [align=left] {$v_{3}$};
    \draw (534,197) node [anchor=north west][inner sep=0.75pt]   [align=left] {$v_{4}$};
    \draw (483,225) node [anchor=north west][inner sep=0.75pt]   [align=left] {$v_{5}$};
    \draw (438,225) node [anchor=north west][inner sep=0.75pt]   [align=left] {$v_{6}$};
    \draw (381, 192) node [anchor=north west][inner sep=0.75pt]   [align=left] {$v_{7}$};
    \draw (364, 149) node [anchor=north west][inner sep=0.75pt]   [align=left] {$v_{8}$};
    \draw (367,100) node [anchor=north west][inner sep=0.75pt]   [align=left] {$v_{9}$};
    \draw (392,58) node [anchor=north west][inner sep=0.75pt]   [align=left] {$v_{10}$};
    \end{tikzpicture}
    \caption{The set of selected vertices (in red) for cases $m=4$ and $m=5$.} 
    \label{Fig:C9C11Case2}
    \end{figure}

    Fix a $(2m+1)$-cycle $C = v_0 v_1 \cdots v_{2m} \subseteq G$. 
    For convenience, we slightly abuse notation by writing $C = \{v_0,\ldots,v_{2m}\}$ when no confusion arises.  
    Fix a vertex $u \in V$ with $d_{G}(u) = \Delta$. 
    By assumption, we have $u \in V\setminus C$. 

    \begin{claim}\label{CLAIM:odd-cycle-max-neighbor}
        There exists $i \in [0,2m]$ such that 
        \begin{align*}
            N_{G}(u) \cap \bigcup_{i\in [0,2m]}N_{G}(v_i)
            \subseteq N_{G}(v_i) \cup N_{G}(v_{i+2}),  
        \end{align*}
        where the indices are taken modulo $2m+1$. 
    \end{claim}
    \begin{proof}[Proof of Claim~\ref{CLAIM:odd-cycle-max-neighbor}]
        Suppose to the contrary that this claim fails. 
        Then there exist $\{i,j\} \subseteq [0, 2m]$ with $|i-j| \not\equiv 2 \Mod{2m+1}$ such that $N_{G}(u) \cap N_{G}(v_i) \neq \emptyset$ and $N_{G}(u) \cap N_{G}(v_j) \neq \emptyset$. 
        Fix $x \in N_{G}(u) \cap N_{G}(v_i)$ and $y \in N_{G}(u) \cap N_{G}(v_j)$. 
        It follows from Fact~\ref{FACT:odd-girth} that $N_{G}(v_i) \cap N_{G}(v_j) = \emptyset$. Thus, $x \neq y$.  
        
        Note that $|i-j| \not\equiv 1 \Mod{2m+1}$, since otherwise $\{v_i, v_j, x, y, u\}$ would span a copy of $C_5$ in $G$, a contradiction. 
        Therefore, the distance between $v_i$ and $v_j$ on $C$ is at least three. 
        Suppose by symmetry that $i = 1$ and $j - i \equiv 1 \Mod{2}$. 
        Then $j \in \{4, 6, \ldots, 2m-2\}$.
        
        If $j \in \{4, 6, \ldots, 2m-4\}$, then $\{v_i, v_{i+1}, \ldots, v_j, y,u,x\}$ spans an odd cycle of length $j - 1 + 4 \le 2m-1$, contradicting the assumption that $g_{\mathrm{odd}} = 2m+1$. 
        
        If $j = 2m-2$, then $\{v_i, v_{i+1}, \ldots, v_j, y,u,x\}$ spans an odd cycle of length $j - 1 + 4 = 2m+1$ with $d_{G}(u) = \Delta$, a contradiction to the assumption that no $(2m+1)$-cycle in $G$ contains a vertex of maximum degree.  
    \end{proof}

    By relabeling the vertices on $C$ if necessary, we may assume that 
    \begin{align*}
        N_{G}(u) \cap \bigcup_{i\in [0,2m]}N_{G}(v_i)
            \subseteq N_{G}(v_1) \cup N_{G}(v_{2m}). 
    \end{align*}
    Let (see Figure~\ref{Fig:C9C11Case2})
    \begin{align*}
        S
        \coloneqq
        \begin{cases}
            \left\{ 2,~3,~6,~7,~\ldots,~2m-2,~2m-1 \right\}, & \quad\text{if}\quad m \equiv 0 \Mod{2}, \\
            \left\{ 0,~3,4,~\ldots,~2m-3,~2m-2 \right\}, & \quad\text{if}\quad m \equiv 1 \Mod{2}. 
        \end{cases}
    \end{align*}
    For every pair $\{i,j\} \subseteq S$, the distance between vertices $v_i$ and $v_j$ along $C$ is not $2$, so it follows from Fact~\ref{FACT:odd-girth} that $N_{G}(v_i) \cap N_{G}(v_j) = \emptyset$.
    Moreover, it follows from Claim~\ref{CLAIM:odd-cycle-max-neighbor} that $N_{G}(u) \cap N_{G}(v_i) = \emptyset$ for every $i \in S$. 
    Therefore, 
    \begin{align*}
        n 
        \ge |N_{G}(u)| + \sum_{i \in S}|N_{G}(v_i)|
        \ge \Delta + m \delta 
        \ge \Delta + (k+1) \delta.
    \end{align*}
    If $\delta > \frac{n}{k+1} - \frac{\Delta}{2k+2}$, then 
    \begin{align*}
        n > \Delta + (k+1)\left( \frac{n}{k+1} - \frac{\Delta}{2k+2} \right)
        = n + \frac{\Delta}{2}, 
    \end{align*}
    a contradiction. 

    If $\delta > \frac{n-1-\Delta}{k}$, then 
    \begin{align*}
        n > \Delta + k \cdot \frac{n-1-\Delta}{k} + \delta 
        = n - 1 + \delta 
        \ge n
    \end{align*}
    also a contradiction.
    This completes the proof of Theorem~\ref{THM:max-deg-AES-odd-cycle}. 
\end{proof}

\section{Concluding remarks}\label{SEC:Remarks}
In this note, we studied the effect of a large-degree vertex on the classical Andr{\'a}sfai--Erd\H{o}s--S\'{o}s theorem.
We hope this will motivate further attention to the broader program of understanding how the presence of a large-degree vertex influences extremal problems.
For example, a natural step beyond Theorem~\ref{THM:max-deg-AES-clique} would be to investigate the results of~\cite{Jin93,CJK97,HJ98,GL11} under an additional maximum-degree constraint.

\section*{Acknowledgments}
X.L. was supported by the Excellent Young Talents Program (Overseas) of the National Natural Science Foundation of China. 
J.W. was supported by the National Natural Science Foundation of China No.~12471316 and Natural Science Foundation of Shanxi Province, China No.~RD2200004810.
\bibliographystyle{alpha}
\bibliography{AES}

\newcommand{\etalchar}[1]{$^{#1}$}
\begin{thebibliography}{HHL{\etalchar{+}}25}

\bibitem[AES74]{AES74}
B.~Andr{\'a}sfai, P.~Erd\H{o}s, and V.~T. S{\'o}s.
\newblock On the connection between chromatic number, maximal clique and minimal degree of a graph.
\newblock {\em Discrete Math.}, 8:205--218, 1974.

\bibitem[BBRS03]{BBRS03}
Paul Balister, B{\'e}la Bollob{\'a}s, Oliver Riordan, and Richard~H. Schelp.
\newblock Graphs with large maximum degree containing no odd cycles of a given length.
\newblock {\em J. Combin. Theory Ser. B}, 87(2):366--373, 2003.

\bibitem[Bol78]{Bol78}
B{\'e}la Bollob{\'a}s.
\newblock {\em Extremal graph theory}, volume~11 of {\em London Mathematical Society Monographs}.
\newblock Academic Press, Inc. [Harcourt Brace Jovanovich, Publishers], London-New York, 1978.

\bibitem[Bra03]{Bra03}
Stephan Brandt.
\newblock On the structure of graphs with bounded clique number.
\newblock {\em Combinatorica}, 23(4):693--696, 2003.

\bibitem[CJK97]{CJK97}
C.~C. Chen, G.~P. Jin, and K.~M. Koh.
\newblock Triangle-free graphs with large degree.
\newblock {\em Combin. Probab. Comput.}, 6(4):381--396, 1997.

\bibitem[DHLY25]{DHLY25}
Jinghua Deng, Jianfeng Hou, Xizhi Liu, and Caihong Yang.
\newblock Tight bounds for rainbow partial {$F$}-tiling in edge-colored complete hypergraphs.
\newblock {\em J. Graph Theory}, 110(4):457--467, 2025.

\bibitem[Erd70]{Erd70}
P{\'a}l Erd{\H o}s.
\newblock On the graph theorem of {T}ur\'an.
\newblock {\em Mat. Lapok}, 21:249--251, 1970.

\bibitem[F{\"u}r15]{Fur15}
Zolt{\'a}n F{\"u}redi.
\newblock A proof of the stability of extremal graphs, {S}imonovits' stability from {S}zemer{\'e}di's regularity.
\newblock {\em J. Combin. Theory Ser. B}, 115:66--71, 2015.

\bibitem[GL11]{GL11}
Wayne Goddard and Jeremy Lyle.
\newblock Dense graphs with small clique number.
\newblock {\em J. Graph Theory}, 66(4):319--331, 2011.

\bibitem[HHL{\etalchar{+}}24]{HHLLYZ23c}
Jianfeng Hou, Caiyun Hu, Heng Li, Xizhi Liu, Caihong Yang, and Yixiao Zhang.
\newblock On the boundedness of degenerate hypergraphs.
\newblock {\em arXiv preprint arXiv:2407.00427}, 2024.

\bibitem[HHL{\etalchar{+}}25]{HHLLYZ25deg}
Jianfeng Hou, Caiyun Hu, Heng Li, Xizhi Liu, Caihong Yang, and Yixiao Zhang.
\newblock Toward a density {C}orr\'adi-{H}ajnal theorem for degenerate hypergraphs.
\newblock {\em J. Combin. Theory Ser. B}, 172:221--262, 2025.

\bibitem[HJ98]{HJ98}
Roland H{\"a}ggkvist and Guoping Jin.
\newblock Graphs with odd girth at least seven and high minimum degree.
\newblock {\em Graphs Combin.}, 14(4):351--362, 1998.

\bibitem[HLL{\etalchar{+}}23]{HLLYZ23}
Jianfeng Hou, Heng Li, Xizhi Liu, Long-Tu Yuan, and Yixiao Zhang.
\newblock A step towards a general density {C}orr{\'a}di--{H}ajnal theorem.
\newblock {\em Canadian Journal of Mathematics}, pages 1--36, 2023.

\bibitem[HY22]{HY22}
Qingyi Huo and Long-Tu Yuan.
\newblock Graphs with large maximum degree containing no edge-critical graphs.
\newblock {\em European J. Combin.}, 106:Paper No. 103576, 10, 2022.

\bibitem[Jin93]{Jin93}
Guo~Ping Jin.
\newblock Triangle-free graphs with high minimal degrees.
\newblock {\em Combin. Probab. Comput.}, 2(4):479--490, 1993.

\bibitem[Man07]{Man07}
Willem Mantel.
\newblock Vraagstuk {XXVIII}.
\newblock {\em Wiskundige Opgaven}, 10(2):60--61, 1907.

\bibitem[Sim68]{Sim69}
M.~Simonovits.
\newblock A method for solving extremal problems in graph theory, stability problems.
\newblock In {\em Theory of {G}raphs ({P}roc. {C}olloq., {T}ihany, 1966)}, pages 279--319. Academic Press, New York-London, 1968.

\bibitem[Tur41]{TU41}
Paul Tur{\'a}n.
\newblock On an extermal problem in graph theory.
\newblock {\em Mat. Fiz. Lapok}, 48:436--452, 1941.

\end{thebibliography}
\end{document}